\documentclass{amsart}

\usepackage[latin1]{inputenc}
\usepackage[T1]{fontenc}
\usepackage{enumerate}
\usepackage{xcolor,pgf,tikz,pgflibraryarrows,pgffor,pgflibrarysnakes}
\usepackage{float}

\newtheorem{theorem}{Theorem}[section]
\newtheorem{lemma}[theorem]{Lemma}
\newtheorem{cor}[theorem]{Corollary}
\newtheorem{prop}[theorem]{Proposition}

\newtheorem{question}{Question}

\theoremstyle{definition}

\newtheorem{example}[theorem]{Example}

\theoremstyle{remark}

\author[Q. Menet]{Quentin Menet}
\address[Q. Menet]{Service de Probabilit\'e et Statistique, D\'epartement de Math\'ematique\\ Universit\'{e} de Mons\\ Place du Parc 20\\ 7000 Mons (Belgium)}
\email{quentin.menet@umons.ac.be}
\author[D. Papathanasiou]{Dimitris Papathanasiou}
\address[D. Papathanasiou]{Service de Probabilit\'e et Statistique, D\'epartement de Math\'ematique\\ Universit\'{e} de Mons\\ Place du Parc 20\\ 7000 Mons (Belgium)}
\email{dimitrios.papathanasiou@umons.ac.be}
 \thanks{The first author is a Research Associate of the Fonds de la Recherche Scientifique - FNRS}

\begin{document}

\title[Accumulation points]{Structure of sets of bounded sequences with a prescribed number of accumulation points}

\maketitle

\begin{abstract}
For each vector $x\in \ell^{\infty}$, we can define the non-empty compact set $L_x$ of accumulation points of $x$. Given an infinite subset $A$ of $\mathbb{N}\backslash\{1\}$, we can therefore investigate under which conditions on $A$, the set $L(A):=\{x\in \ell^\infty: |L_x|\in A\}$ is lineable or even densely lineable. In particular, we show that if $L(A)$ is lineable then there exists $k\ge 1$ such that $A\cap (A-k)$ is infinite and that if $L(A)$ is densely lineable then $A\cap (A-1)$ is infinite. We end up by answering an open question on the existence of a closed non-separable subspace in which each non-zero vector has countably many accumulation points. 
\end{abstract}

\section{Introduction}

Let $\mathbb{N}=\{1,2,3,\dots\}$ and $\ell^{\infty}=\ell^{\infty}(\mathbb{N})$ the space of bounded real sequences endowed with $\|\cdot\|_{\infty}$. It is obvious that for any $x\in \ell^{\infty}$, the set $L_x$ of accumulation points of $x$ is a non-empty compact set, where $a$ is said to be an accumulation point of $x$ if there exists an increasing sequence $(n_k)_{k\ge 1}$ such that $x_{n_k}$ tends to $a$ when $k$ tends to infinity. It follows from the Cantor-Bendixson Theorem that the cardinality of $L_x$, denoted by $|L_x|$, has to belong to $\mathbb{N}\cup\{\omega\}\cup\{\mathfrak{c}\}$ where $\omega$ is the cardinality of $\mathbb{N}$ and $\mathfrak{c}$ the cardinality of $\mathbb{R}$. Given a subset $A$ of $\mathbb{N}\cup\{\omega\}\cup\{\mathfrak{c}\}$, one can thus wonder which properties are satisfied by the set
\[L(A)=\{x\in \ell^{\infty}: |L_x|\in A\}.\]

If we were working in the context of a Baire space, we could be interested in knowing when these sets are comeager or not. In the context of $\ell^{\infty}$, we will be more interested in the linear structure inside $L(A)\cup \{0\}$, i.e. by the lineability, the dense lineability and the spaceability of $L(A)$. Recall that if $X$ is a topological vector space and $Y$ is a subset of $X$, we say that $Y$ is lineable (resp. densely lineable) if $Y\cup \{0\}$ contains an infinite-dimensional subspace (resp. a dense infinite-dimensional subspace). In the same context, we say that $Y$ is spaceable if $Y\cup\{0\}$ contains a closed infinite-dimensional subspace. The terminology of lineability and spaceability has been introduced by Gurariy \cite{Gur1, Gur2} and has been the object of many studies \cite{Aron, Bernal, Bernal2, Car, LeonMontes2, LeonMontes, Men}.

In the case of the sets $L(A)$, the first study was led by the second author. Papathanasiou proved in \cite{Dim} that $\ell^{\infty}\backslash c_0$ is densely lineable and it follows from his proof that $L(\mathfrak{c})$ is in fact densely lineable. Leonetti, Russo and Somaglia~\cite{LRS} have then looked at the other subsets of $\mathbb{N}\cup\{\omega\}\cup\{\mathfrak{c}\}$ and got several interesting results on the structure of $L(A)$. They have proved that $L(\mathfrak{c})$ is also spaceable, that $L(\omega)$ is densely lineable and spaceable and that $L(\mathbb{N}\backslash\{1\})$ is densely lineable but not spaceable.
When $A$ is a subset of $\mathbb{N}\backslash\{1\}$, we can easily show that if $A$ is cofinite then $L(A)$ is lineable. On the other hand, if $A$ is a finite subset of $\mathbb{N}\backslash\{1\}$ then $L(A)$ is not lineable \cite[Theorem 4.4]{LRS}. If $A$ is an infinite subset of $\mathbb{N}\backslash\{1\}$ which is not cofinite, Leonetti, Russo and Somaglia proved that both behaviors can appear. For instance, $L(2\mathbb{N}+1)$ is lineable while $L(\{3^n:n\ge 1\})$ is not lineable. In fact, they have shown that if $a_0\ge 2$ and $a_{n+1}>2a_n$ then $L(\{a_n:n\ge 0\})$ is not lineable. These different results have led to different open questions (see \cite{LRS}).

\begin{question}
Is $L(2\mathbb{N})$ lineable? Is $L(\{n^2:n\ge 2\})$ lineable?
\end{question}

\begin{question}\label{Q2}
Is $L(2\mathbb{N}+1)$ densely lineable in $\ell^{\infty}$?
\end{question}

\begin{question}\label{Q3}
Does $L(\omega)$ or $L(\mathbb{N}\cup\{\omega\})$ contain a non-separable closed subspace? Does it contain an isometric copy of $\ell^{\infty}$?
\end{question}

Excepted the case of $2\mathbb{N}$, we answer all these questions in the following sections. In Section~\ref{lin}, we focus on the lineability and the dense lineability of $L(A)$ for some subset $A\subset \mathbb{N}\backslash\{1\}$. More precisely, we prove that if $A\subset \mathbb{N}\backslash\{1\}$ and $L(A)$ is lineable then there exists $k\ge 1$ such that 
\[|A\cap (A-k)|=\infty.\]
This allows in particular to deduce that $L(\{n^2:n\ge 2\}$ is not lineable. We then prove that  if $A\subset \mathbb{N}\backslash\{1\}$ and $L(A)$ is densely lineable in $\ell^{\infty}$ then
\[|A\cap (A-1)|=\infty.\]
This answers Question~\ref{Q2} negatively. Finally, we investigate the closed infinite-dimensional subspaces of $L(\omega)$ in Section~\ref{spa} and we show that although $L(\mathbb{N}\cup\{\omega\})$ does not contain an isometric copy of $\ell^{\infty}$, the set $L(\omega)$ (and thus also $L(\mathbb{N}\cup\{\omega\}$) contains a non-separable closed subspace, answering Question~\ref{Q3}.

\section{Lineability and dense lineability}\label{lin}

Leonetti, Russo and Somaglia~\cite{LRS} have already proved that $L(\mathbb{N}\backslash\{1\})$ and $L(\omega)$ are both densely lineable. They have also remarked that it follows from the results of Papathanasiou~\cite{Dim} that $L(\mathfrak{c})$ is densely lineable. For these reasons, we will restrict our study of lineability and dense lineablity to subsets $A$ of $\mathbb{N}\backslash\{1\}$.

Given two sequences $x,y\in \ell^{\infty}$ with a finite number of accumulation points bigger than $2$, we would like to better understand the possibilities of cardinality for the set $L_{\lambda x+\mu y}$ for any non-zero couple of real numbers $(\lambda, \mu)$. We will use in this section the notations introduced in \cite{LRS}. In particular, as mentioned in \cite{LRS}, we can always find a partition $S_1,\dots,S_n$ of $\mathbb{N}$ in infinite sets and distinct real numbers $\xi_1,\dots,\xi_n$ such that $x\sim_{c_0} \xi_1 1_{S_1}+\dots +\xi_n 1_{S_n}$, \emph{i.e.} $x-\xi_1 1_{S_1}+\dots +\xi_n 1_{S_n}\in c_0$ where $1_{S}=(z_n)_n$ with $z_n=1$ if $n\in S$ and $z_n=0$ otherwise. By considering a partition $T_1,\dots,T_k$ of $\mathbb{N}$ in infinite sets and distinct real numbers $\eta_1,\dots,\eta_k$ such that $y\sim_{c_0} \eta_1 1_{T_1}+\dots +\eta_k 1_{T_k}$, we can easily state that 
\[L_{\lambda x+\mu y}=\{\lambda \xi_i +\mu \eta_j: S_i\cap T_j \ \text{is infinite},\ 1\le i\le n,\ 1\le j\le k\}.\]
In view of this equality, we introduce the following notations : 
$$
\mathcal{E}_{x,y}=\{ (i,j)\in \{1,\dots ,n\}\times \{1,\dots ,k\}: S_i\cap T_j \,\, \text{is infinite}\}
$$
and
$$
\mathcal{P}_{x,y}=\{ (\xi_i,\eta_j): (i,j)\in \mathcal{E}_{x,y}\}.$$
The ideas of several proofs in this paper come from the study of the configuration of $\mathcal{P}_{x,y}$ in $\mathbb{R}^2$. For instance, two elements $(\xi_i,\eta_j)\in \mathcal{P}_{x,y}$ and $(\xi_{i'},\eta_{j'})\in \mathcal{P}_{x,y}$ will give the same accumulation points for $\lambda x+\mu y$ if and only if their images by $x^*(a,b)=\lambda a+\mu b$ coincide. Geometrically, this means that $(\xi_i,\eta_j)$ and $(\xi_{i'},\eta_{j'})$ are on the same line of equation $x^*(a,b)=c$ for some $c$. The cardinality of $L_{\lambda x+\mu y}$ can thus be interpreted as the minimal number of real numbers $c_1,\dots,c_m$ such that the parallel lines given by $x^*(a,b)=c_l$ cover $\mathcal{P}_{x,y}$. For instance, since it is always possible to find a line passing through the origin, non-parallel to any of the (finitely many) line segments with vertices from $\mathcal{P}_{x,y}$, we get the existence of a vector $z\in \text{span}\{x,y\}$ such that $|L_{z}|=|\mathcal{E}_{x,y}|$ and we remark that for any $z\in \text{span}\{x,y\}$, $|L_z|\le |\mathcal{E}_{x,y}|$. On the other hand, by considering a line parallel to at least one line segment joining two vertices from $\mathcal{P}_{x,y}$, we get a vector $z\in \text{span}\{x,y\}$ such that $|L_{z}|<|\mathcal{E}_{x,y}|$. We state these simple facts in the following lemma for further reference.

\begin{lemma}\label{obvious}
Let $x\in L(n)$ and $y\in L(k)$ for some $n,k\ge 2$. There exist $z_1,z_2\in \text{\emph{span}}\{x,y\}$ satisfying 
$$
|L_{z_1}|=|\mathcal{E}_{x,y}|\quad\text{and}\quad |L_{z_2}|<|\mathcal{E}_{x,y}|
$$
and 
\[|\mathcal{E}_{x,y}|=\max\{|L_z|:z\in \text{\emph{span}}\{x,y\}\}.\]
\end{lemma}

As expected, if $A$ is a finite subset of $\mathbb{N}\backslash\{1\}$, $L(A)$ cannot be lineable but if $A$ is infinite, $L(A)$ can be lineable even if $A$ is not cofinite. It is for instance the case for $2\mathbb{N}+1$ (see \cite{LRS}). However, if we denote by $(a_n)$ the increasing enumeration of $A$ then it is proved in \cite{LRS} that if $a_0\ge 2$ and $a_{n+1}>2a_n$ then $L(A)$ cannot be lineable. Leonetti, Russo and Somaglia then asked if $L(A)$ is lineable when $A=\{n^2:n\ge 2\}$. We therefore generalize the above condition to be able to deal with the set $\{n^2:n\ge 2\}$. The proof of the below proposition will rely on the following idea : if the set $\mathcal{P}_{x,y}$ contains $n_2$ points spread over $n_1$ columns then the smallest slope appearing between two points of consecutive columns can appear at most $n_1-1$ times so that there exists $z\in \text{span}\{x,y\}$ satisfying $n_2-n_1<|L_z|<n_2$.

\begin{prop}\label{diff}
Let $A\subset \mathbb{N}\backslash\{1\}$.
If $L(A)\cup\{0\}$ contains a subspace $M$ of dimension $2$ and if we let 
\[n_1=\min\{|L_z|:z\in M\backslash\{0\}\}\in A \quad\text{and}\quad n_2=\max\{|L_z|:z\in M\}\in A,\]
then
\[ ]n_2-n_1,n_2[\cap A\ne \emptyset.\]
\end{prop}
\begin{proof}
Let $M$ be a 2-dimensional space in $L(A)\cup\{0\}$. Let
\[n_1=\min\{|L_z|:z\in M\backslash\{0\}\}\in A \quad\text{and}\quad n_2=\max\{|L_z|:z\in M\}\in A.\]
 We choose linearly independent vectors $x,y\in M$ such that $|L_x|=n_1$ and $|L_y|=n_2$. It follows from Lemma~\ref{obvious} that $n_2=|\mathcal{E}_{x,y}|$ and that $n_2>n_1$. We now want to show that there exists $z\in \text{span}\{x,y\}$ such that 
 $$
 |L_z|\in ]n_2-n_1,n_2[.
 $$
To this end, we write
$$
x\sim_{c_0}a_1 1_{S_1}+\cdots+a_{n_1 }1_{S_{n_1}}
$$ 
with $S_1,\dots,S_{n_1}$ a partition of $\mathbb{N}$ in infinite sets and $a_1<a_2<\cdots <a_{n_1}$, and 
$$
y\sim_{c_0}b_1 1_{T_1}+\cdots+b_{n_2} 1_{T_{n_2}}
$$ 
with $T_1,\dots,T_{n_2}$ a partition of $\mathbb{N}$ in infinite sets and  $b_1<b_2<\cdots <b_{n_2}$. For $1\le l\le n_1$, we consider $i_l$ to be the smallest integer in $[1,n_2]$ such that $|S_l\cap T_{i_l}|=\infty$, and  $I_l$ to be the biggest integer in $[1,n_2]$ such that $|S_l\cap T_{I_l}|=\infty$.
We look at 
\[\mathcal{C}=\{\frac{b_{i_{l+1}}-b_{I_l}}{a_{l+1}-a_{l}}: 1\le l\le n_1-1\}.\]
Let $C=\min{\mathcal{C}}$ and $l_1,\cdots, l_r$, $1\leq r<n_1$ be an enumeration of indices such that $$
C=\frac{b_{i_{l_s+1}}-b_{I_{l_s}}}{a_{l_s+1}-a_{l_s}},\quad 1\leq s\leq r.
$$
We show that $z=-Cx+y$ satisfies $|L_z|\in ]n_2-n_1,n_2[$. Indeed, if we let $A_l=\{-C a_l + b_i:|S_l\cap T_i|=\infty,\ 1\le i\le n_2\}$ then we have that
\[L_z=\bigcup_{1\le l\le n_1-1} A_l.\]
Since $b_1<b_2<\cdots<b_{n_2}$, we get that $\sum_{1\le l\le n_1-1} |A_l|=|\mathcal{E}_{x,y}|$. Moreover, by definition of $C$, we have
\[\max A_l=-C a_l + b_{I_l}\le -C a_{l+1} + b_{i_{l+1}}  = \min A_{l+1}.\]
Finally, we remark that the above inequality is an equality exactly when $l=l_s$ for some $1\le s\le r$ and thus
\[|\mathcal{E}_{x,y}|-r=\left|\bigcup_{1\le l\le n_1-1} A_l\right|= |L_z|.\]
Since $1\le r\le n_1-1$, we can conclude because $z\in M\backslash\{0\}$ and thus $|L_z|\in A$.
\end{proof}

The above proposition and the fact that every infinite-dimensional subspace $M$ in $L(\mathbb{N}\backslash\{1\})\cup\{0\}$ contains vectors with an arbitrarily big amount of accumulation points lead us to the following result.

\begin{theorem}
Let $A\subset \mathbb{N}\backslash\{1\}$. If $L(A)$ is lineable then there exists $k\ge 1$ such that 
\[|A\cap (A-k)|=\infty.\]
\end{theorem}
\begin{proof}
Let $M$ be an infinite-dimensional space in $L(A)\cup\{0\}$. We consider $x\in M$ such that $|L_x|=\min\{|L_z|:z\in M\}$ and we set $K=|L_x|$. There also exists a sequence $(y_n)$ in $M$ with $\lim_{n}|L_{y_n}|=\infty$ since $L(F)$ is not lineable when $F$ is a finite subset of $\mathbb{N}\backslash\{1\}$ (\cite[Theorem 4.4]{LRS}). By applying Proposition~\ref{diff} to $\text{span}\{x,y_n\}$, we deduce that \[]|\mathcal{E}_{x,y_n}|-K,|\mathcal{E}_{x,y_n}|[\cap A\ne \emptyset.\] In particular, there exist $1\le k\le K$ and an increasing sequence $(n_j)$ such that $|\mathcal{E}_{x,y_{n_j}}|-k\in A$ for any $j$.  Since $|\mathcal{E}_{x,y_{n_j}}|\ge |L_{y_{n_j}}|$ and $|\mathcal{E}_{x,y_{n_j}}|\in A$, we get that $|A\cap (A-k)|=\infty$.
\end{proof}

This theorem allows us to answer an open question posed in \cite{LRS} by asserting that $L(\{n^2:n\ge 2\})$ is not lineable. More generally, we get the following example:
\begin{example}\label{ex1}
If $(a_n)_{n\ge 0}$ is an increasing sequence with $a_0\ge 2$ and $\lim_n (a_{n+1}-a_{n})=\infty$ then $L(\{a_n:n\ge 0\})$ is not lineable.
\end{example}

In view of the previous example and the case of $2\mathbb{N}+1$, one can wonder if $L(\{a_n:n\ge 0\})$ can be lineable when $\limsup_n (a_{n+1}-a_{n})=\infty$ but $\liminf_n (a_{n+1}-a_{n})<\infty$. 
In order to get such an example, we will use the following lemma which is obtained by applying successively Lemma 2.2 from \cite{LRS}.
 
\begin{lemma} \label{estimate}
If $x_1, \dots ,x_k\in L(\mathbb{N})$ satisfy that $|L_{x_i}|=n_i$, $1\leq i\leq k$, and $z=\sum_{i=1}^ka_ix_i$ with $a_k\neq 0$, then 
$$
\frac{n_k}{n_1\dots n_{k-1}}\leq |L_z|\leq n_1\dots n_k.
$$
\end{lemma}

We now show that there exists a set $A$ with arbitrarily big gaps such that $L(A)$ is lineable.

\begin{prop}\label{nk}
Let $(n_k)_{k\ge 1}$ be an increasing sequence. There exists an infinite set $\mathcal{K}$ such that the set 
$$
L\left([2,+\infty[\backslash\bigcup_{k\in\mathcal{K}}[n_k,n_{k+1}[\right)
$$ 
is lineable.
\end{prop}

\begin{proof}
We seek to construct by induction a linearly independent family $(x_n)_{n\ge 1}\subset \ell^{\infty}$ and an increasing sequence of integers $(k_j)_{j\ge 1}$ such that 
\[\text{span}\{x_n: n\in \mathbb{N}\}\subset L\left([2,+\infty[\backslash\bigcup_{k\in\mathcal{K}}[n_k,n_{k+1}[\right)\cup \{0\}\] for $\mathcal{K}=\{k_j:j\ge 1\}$.

Let $x_1\in \ell^{\infty}$ be a sequence with $|L_{x_1}|:=l_1=2$ and $k_1$ be such that $n_{k_1}>l_1$. We choose $l_2$ such that $\frac{l_2}{l_1}\geq n_{k_1+1}$ and then a sequence $x_2\in \ell^{\infty}$ with $|L_{x_2}|=l_2$ and $k_2$ such that $n_{k_2}>l_1 l_2$. Inductively, if $r\geq 1$, we choose $l_r$ such that
$$
\frac{l_r}{l_1\dots l_{r-1}}\geq n_{k_{r-1}+1}
$$
and then a sequence $x_r\in \ell^{\infty}$ such that $|L_{x_r}|=l_r$ and $k_r$ such that
$$
n_{k_r}>l_1\dots l_r.
$$
Then $M:=\text{span}\{x_n: n\in \mathbb{N}\}$ is an infinite-dimensional subspace contained in 
$$
L\left([2,+\infty[\backslash\bigcup_{k\in\mathcal{K}}[n_k,n_{k+1}[\right)\cup \{0\}
$$ 
for $\mathcal{K}:=\{k_j: j\in \mathbb{N}\}$. Indeed, if $z=\sum_{i=1}^ja_ix_i \in M$ with $a_j\neq 0$, then by Lemma~\ref{estimate} and the induction process, we get that
$$
n_{k_{j-1}+1}\leq \frac{l_j}{l_1\dots l_{j-1}}\leq |L_z|\leq l_1\dots l_j<n_{k_j}.
$$

\end{proof}

In particular, if we apply the previous proposition with a sequence $(n_k)$ satisfying $\lim_k(n_{k+1}-n_k)=\infty$, we deduce that there exists an increasing sequence $(a_n)$ with $a_0\ge 2$ and $\limsup_n (a_{n+1}-a_{n})=\infty$ such that $L(\{a_n:n\ge 0\})$ is lineable.\\

Another question posed on \cite{LRS} concerns the dense lineability of $L(2\mathbb{N}+1)$. Indeed, they prove there that $L(2\mathbb{N}+1)$ is lineable and that $L(\mathbb{N}\backslash\{1\})$ is densely lineable. It is thus natural to try to determine if $L(2\mathbb{N}+1)$ can also be densely lineable. The answer turns out to be negative as infinite sets $A$ such that $L(A)$ is densely lineable have to satisfy the stronger assumption that $|A\cap (A-1)|=\infty$. The idea of this result can be summarized as follows. If $y$ is a non-zero sequence whose accumulation points are separated by $\delta>0$ and if $x$ is a sequence really close to $1_{2\mathbb{N}}-1_{2\mathbb{N}+1}$ then the set $\mathcal{P}_{x,y}$ will belong to two thin rectangles $\{(\xi,\eta):|\xi-1|<\varepsilon\}$ and $\{(\xi,\eta):|\xi+1|<\varepsilon\}$ around $(1,0)$ and $(-1,0)$. If $\varepsilon$ is sufficiently small compared to $\delta$ and $\|y\|^{-1}$ then the slopes between two points of the same rectangle will have a modulus equals to $0$ or farther from $0$ than the slopes between two points of the two rectangles. It follows that the biggest slope (or the smallest if the biggest is equal to $0$) existing between the two rectangles will only appear once, implying that there exists $z\in \text{span}\{x,y\}$ such that $|L_z|=|\mathcal{E}_{x,y}|-1$.

\begin{theorem}\label{A-1}
Let $A\subset \mathbb{N}\backslash\{1\}$. If $L(A)$ is densely lineable then
\[|A\cap (A-1)|=\infty.\]
\end{theorem}
\begin{proof}
Let $M$ be a dense, infinite-dimensional vector space in $L(A)\cup\{0\}$ and let $y\in M$ be a non-zero sequence. We can write
$$
y\sim_{c_0} a_1 1_{S_1}+\cdots+a_N 1_{S_N}
$$ 
where $S_1,\dots, S_N$ is a partition of $\mathbb{N}$ in infinite sets and $a_1<a_2<\cdots <a_N$. We consider 
$$
\delta:=\min\{a_{l+1}-a_l:1\le l<N\}.
$$ 
Since $M$ is dense, there exists $x\in M$ such that 
$$
\|x-(1_{2\mathbb{N}}-1_{2\mathbb{N}+1})\|<\varepsilon=\min\left\{\frac{\delta}{8\|y\|},\frac{1}{2}\right\}.
$$
We can then write 
$$
x\sim_{c_0}(b_1 1_{T_1}+\cdots+b_L 1_{T_L})+ (c_1 1_{Q_1}+\cdots+c_J 1_{Q_J})
$$ 
with \[-1-\varepsilon\le b_1<...<b_L\le -1+ \varepsilon <1-\varepsilon\le  c_1<...<c_J\le 1+\varepsilon.\]
We let 
\[C:=\max\{\frac{a_k-a_i}{c_j-b_l}~:~|T_l\cap S_i|=\infty,\ |Q_j\cap S_k|=\infty, 1\le j\le J,\ 1\le l\le L,\ 1\le k,i\le N\}.\]
Note that if $C=0$, we can replace $y$ by $-y$ to get $C\ne 0$ (since the elements in $\mathcal{P}_{x,y}$ cannot be all aligned because $1\notin A$). We can thus assume without loss of generality that we have $C\ne 0$ and we remark that
\[|C|\le \frac{\|y\|}{1-\varepsilon}\le 2\|y\|.\] 
We now consider 
$$
z=-Cx+y
$$ 
and we prove that $|L_z|=|\mathcal{E}_{x,y}|-1$.

Let 
$$
A_l=\{-C b_l + a_i:|T_l\cap S_i|=\infty,\ 1\le i\le N\}
$$ for $1\le l\le L$
and 
$$
B_j=\{-C c_j + a_k:|Q_j\cap S_k|=\infty, \ 1\le i\le N\}
$$
for $1\le j\le J$.
We know that we have 
\[L_z=\bigcup_{1\le l\le L} A_l \cup \bigcup_{1\le j\le J} B_j\quad \text{and}\quad \sum_{1\le l\le L}|A_l|+\sum_{1\le j\le J}|B_j|=|\mathcal{E}_{x,y}|.\]
Pick $i_0,j_0,k_0,l_0$ such that $|T_{l_0}\cap S_{i_0}|=\infty$, $|Q_{j_0}\cap S_{k_0}|=\infty$ and
$$
C=\frac{a_{k_0}-a_{i_0}}{c_{j_0}-b_{l_0}}.
$$
We show that among all the elements in the sets $A_l$ and the sets $B_j$, there is only two equal elements $-C b_{l_0}+ a_{i_0}$ and $-Cc_{j_0}+a_{k_0}$. It will therefore follow that $|L_z|=|\mathcal{E}_{x,y}|-1$.

We have three cases to consider:
\begin{enumerate}
\item Let $1\le l<l'\le L$ and $1\le i,i'\le N$. If $-C b_l + a_i=-C b_{l'}+a_{i'}$ then
$C=\frac{a_{i'}-a_i}{b_{l'}-b_l}$. However, we know that $C\ne 0$ and thus $i'\ne i$. It will then follow by definition of $\delta$ and $\varepsilon$ that 
\[|C|\ge \frac{\delta}{2\varepsilon}\ge 4\|y\|.\]
This is a contradiction since $|C|\le 2\|y\|$.
\item  Let $1\le j<j'\le J$ and $1\le k,k'\le N$. If $-C c_j + a_k=-C c_{j'}+a_{k'}$ then as previously, we get
\[|C|=\frac{|a_{k'}-a_k|}{|c_{j'}-c_j|}\ge \frac{\delta}{2\varepsilon}\ge 4\|y\|,\]
which is a contradiction.
\item Let $1\le j\le J$, $1\le l\le L$ and $1\le k,k'\le N$. If $-C c_j + a_k=-C b_{l}+a_{i}$ with $-C c_j + a_k\in B_j$ and $-C b_{l}+a_{i}\in A_l$ then
$$
-Cc_{j_0}+a_{k_0}=-Cb_{l_0}+a_{i_0}= -Cb_{l}+a_{i}=-Cc_{j}+a_{k}.
$$ 
Indeed, if we had for instance
$$
-Cc_{j_0}+a_{k_0}=-Cb_{l_0}+a_{i_0}> -Cb_{l}+a_{i}=-Cc_{j}+a_{k}
$$
 then 
\[\frac{a_{k_0}-a_{i}}{c_{j_0}-b_{l}}> C\]
which is a contradiction. 
It follows that $a_{i_0}-Cb_{l_0}= a_{i}-Cb_{l}$ and $a_{k_0}-Cc_{j_0}=a_{k}-Cc_{j}$. By the previous points, this implies that $l=l_0$ and $j=j_0$ (and thus also that $i=i_0$ and $k=k_0$).
\end{enumerate}
We conclude that $|L_z|=|\mathcal{E}_{x,y}|-1$ and the proof is complete since $|\mathcal{E}_{x,y}|-1\in A\cap (A-1)$.
\end{proof}

As mentioned, it directly follows from this result that $L(2\mathbb{N}+1)$ is not densely lineable. We can also appraise the difference between the lineability and the dense lineabilty through the comparison between Proposition~\ref{nk} and the following result.

\begin{prop}\label{nk2}
If $(n_k)_{k\ge 1}$ satisfies $n_{k+1}>n^2_k$ for any $k\ge 1$ then for any infinite set $\mathcal{K}$, the set 
$$
L\left([2,+\infty[\backslash\bigcup_{k\in\mathcal{K}}[n_k,n_{k+1}[\right)
$$ 
is not densely lineable.
\end{prop}
\begin{proof}
Let $(n_k)_{k\ge 1}$ such that 
$$
n_{k+1}>n^2_k
$$ 
and $\mathcal{K}$ an infinite set. We want to show that if 
$$
M\subset L\left([2,+\infty[\backslash\bigcup_{k\in\mathcal{K}}[n_k,n_{k+1}[\right)\cup\{0\}
$$ 
is a linear subspace, then its dimension is countable.\\

Let $k\in \mathcal{K}$. We show to this end that each family of linearly independent vectors in $M\cap L([2,n_k])$ has cardinality less than $n_k$. Indeed, it is shown in \cite{LRS} that the set $L([2,n_k])$ is exactly $(n_k-1)$-lineable. Hence, if there exists a family $\{x_1,\dots ,x_{n_k}\}$ of linearly independent vectors in $M\cap L([2,n_k])$, then there is a vector $z_0\in \text{span}\{x_1,\dots ,x_{n_k}\}$ with $|L_{z_0}|\geq n_k+1$. Now, we may write
$$
z_0=\sum_{i=1}^{n_k}a_ix_i
$$
where $a_i\in \mathbb{R}$, for $1\leq i\leq n_k$. Set 
$$
z_1=\sum_{i=1}^{n_k-1}a_ix_i.
$$
If $|L_{z_1}|\leq n_k$, then since $|L_{x_{n_k}}|\leq n_k$ as well, we get by Lemma~\ref{estimate} that 
$$
|L_{z_0}|\leq n_k^2.
$$
Otherwise, it holds that 
$$
|L_{z_1}|\geq n_k+1.
$$
Repeating the procedure, we find a vector
$z_i\in \text{span}\{x_1,\dots ,x_{n_k-i}\}$, for some $0\leq i<n_k$, satisfying that
$$
n_k+1\le |L_{z_i}|\leq n_k^2.
$$
But that means that 
$$
M\cap L([n_k+1,n_k^2])\ne\emptyset
$$
which is a contradiction, since $n_{k+1}>n^2_{k}$ and $k\in \mathcal{K}$. Now, this concludes the proof since, any Hamel basis of $M$ will be included in 
$$
\bigcup_{k\in \mathcal{K}}L([2,n_k])
$$
hence it will be countable. This prevents $M$ from being dense in $\ell^{\infty}$.
\end{proof}

From  Proposition~\ref{nk} and  Proposition~\ref{nk2} we get a big family of infinite sets $A$ such that $L(A)$ is lineable and not densely lineable.

\begin{cor}
For any $(n_k)_{k\ge 1}$ satisfying $n_{k+1}>n^2_k$ for any $k\ge 1$, there exists an infinite set $\mathcal{K}$ such that the set 
$$
L\left([2,+\infty[\backslash\bigcup_{k\in\mathcal{K}}[n_k,n_{k+1}[\right)
$$ 
is lineable but not densely lineable.
\end{cor}

Another example of infinite sets $A$ satisfying that $L(A)$ is lineable but not densely lineable could be given by the set of even numbers. Indeed, we can deduce from Theorem~\ref{A-1} that $L(2\mathbb{N})$ is not densely lineable. However, it is an open question to know if $L(2\mathbb{N})$ is lineable. This question can be found in \cite{LRS} and we state it again here because it seems to us very relevant and intriguing.

\begin{question}
Does $L(2\mathbb{N})\cup \{0\}$ contain an infinite-dimensional subspace?
\end{question}

\section{Non-separable spaceability}\label{spa}

Leonetti, Russo and Somaglia have shown in \cite{LRS} that while the set $L(\mathbb{N}\backslash\{1\})$ is not spaceable,  the sets $L(\omega)$ and $L(\mathfrak{c})$ are both spaceable. They have also remarked that $L(\mathfrak{c})\cup\{0\}$ even contains an isometric copy of $\ell^{\infty}$. However, it is not clear if an isometric copy of $\ell^{\infty}$ can also be found inside $L(\omega)\cup\{0\}$ or even if $L(\omega)\cup\{0\}$ contains a closed non-separable subspace. We answer in this section to this open question stated in \cite{LRS}. To this end, we first show that $L(\omega)\cup\{0\}$ and even $L(\mathbb{N}\cup\{\omega\})\cup\{0\}$ does not contain an isometric copy of $\ell^{\infty}$.

\begin{prop}
The set $L(\mathbb{N}\cup\{\omega\})\cup\{0\}$ does not contain an isometric copy of $\ell^{\infty}$.
\end{prop}
\begin{proof}
Assume that there exists a linear isometry $\phi:\ell^{\infty}\to L(\mathbb{N}\cup\{\omega\})\cup\{0\}$. Let $A$ and $B$ be non-empty disjoint subsets of $\mathbb{N}$ and let $P_n(x)$ be the $n$th coordinate of $x$. Since $\|\phi(1_{A})\|=\|\phi(1_{B})\|=\|\phi(1_{A}+1_{B})\|=\|\phi(1_{A}-1_{B})\|=1$, we first remark that if $|P_n(\phi(1_{A}))|=1$ then $P_n((\phi(1_{B}))=0$ and if there exists an increasing sequence $(n_k)$ such that $\lim_k |P_{n_k}(\phi(1_{A}))|=1$ then $\lim_k P_{n_k}(\phi(1_{B}))=0$. 

We now show that for any infinite set $A$ in $\mathbb{N}$, there exists an increasing sequence $(n_k)$ such that $\lim_k|P_{n_k}\phi(1_A)|=1$. Assume that it is not the case. Then, since $\|\phi(1_A)\|=1$, there exists a non-empty finite set $\mathcal{N}$ such that for any $n\in \mathcal{N}$, $|P_n\phi(1_A)|=1$ and for any $n\notin \mathcal{N}$, $|P_n\phi(1_A)|<1$. Let $N=|\mathcal{N}|$ and $A_{1},\dots,A_{N+1}$ be a partition of $A$ in infinite sets. We have that the sets $\mathcal{N}_k=\{n:|P_n\phi(1_{A_k})|=1\}$ are disjoint and $\mathcal{N}_k\subset \mathcal{N}$ for any $1\le k\le N+1$. 
By definition of $N$, it implies that there exists $1\le k\le N+1$ such that $A_k$ is empty. This means that there exists an increasing sequence $(n_j)$ such that $|P_{n_j}(\phi(1_{A_k}))|$ tends to $1$ and therefore $|P_{n_j}(\phi(1_{A}))|$ tends to $1$, which is a contradiction.

In other words, we have shown that for any infinite set $A$ in $\mathbb{N}$, $1$ or $-1$ belongs to $L_{\phi(1_{A})}$. Note that if $B$ is an infinite subset of $A$ and $1\in L_{\phi(1_{B})}$ (resp. $-1\in L_{\phi(1_{B})}$) then $1\in L_{\phi(1_{A})}$ (resp. $-1\in L_{\phi(1_{A})}$). There then exists an infinite set $A$ in $\mathbb{N}$ such that $1\in L_{\phi(1_{B})}$ for any infinite subset $B$ of $A$ or $-1\in L_{\phi(1_{B})}$ for any infinite subset $B$ of $A$. Indeed, it suffices to consider $A$ such that $1\in L_{\phi(1_{A})}$ and $-1\notin L_{\phi(1_{A})}$ or such that $-1\in L_{\phi(1_{A})}$ and $1\notin L_{\phi(1_{A})}$. If such a set does not exist, it means that for any infinite set the real numbers $1$ and $-1$ belong to $L_{\phi(1_{A})}$ and then any infinite set $A$ will satisfy the desired property.

Without loss of generality, assume that there exists an infinite set $A$ such that for any infinite subset $B$ of $A$, $1\in L_{\phi(1_{B})}$. We consider a family $(A_{j,k})_{j\ge 1,0\le k\le 2^j-1}$ of infinite subsets of $A$ such that $A_{j,k}\cap A_{j,k'}=\emptyset$ if $k\ne k'$ and $A_{j,k}=A_{j+1,2k}\cup A_{j+1,2k+1}$. We let
\[y_j=\sum_{k=0}^{2^{j}-1}\frac{k}{2^j}1_{A_{j,k}}.\]
The sequence $(y_j)_{j\ge 1}$ is convergent in $\ell^{\infty}$ because
\begin{align*}
y_{j+1}-y_j&=\sum_{k=0}^{2^{j+1}-1}\frac{k}{2^{j+1}}1_{A_{j+1,k}}-\sum_{k=0}^{2^{j}-1}\frac{k}{2^j}1_{A_{j,k}}\\
&=\sum_{k=0}^{2^{j}-1}(\frac{2k}{2^{j+1}}-\frac{k}{2^j})1_{A_{j+1,2k}}+
\sum_{k=0}^{2^{j}-1}(\frac{2k+1}{2^{j+1}}-\frac{k}{2^j})1_{A_{j+1,2k+1}}\\
&=\sum_{k=0}^{2^{j}-1}\frac{1}{2^{j+1}}1_{A_{j+1,2k+1}}
\end{align*}
and thus $\|y_{j+1}-y_j\|=\frac{1}{2^{j+1}}$. We denote by $y$ the limit of $(y_j)_{j\ge 1}$ and we get that the sequence $(\phi(y_j))_{j\ge 1}$ tends to $\phi(y)$.

By properties of $A$, we can remark that for any $j\ge 1$, we have
\[L_{\phi(y_j)}\supset \{\frac{k}{2^j}:0\le k\le 2^j-1\}\] since for any $0\le k\le 2^j-1$, there exists an increasing sequence $(n_l)$ such that $P_{n_l}(\phi(1_{A_{j,k}}))$ tends to $1$ and $P_{n_l}(\phi(1_{A_{j,k'}}))$ tends to $0$ when $k'\ne k$. This allows us to conclude that $L_{\phi(y)}$ contains $[0,1]$ which is not countable. Indeed, $L_{\phi(y)}$ is a closed set and for any $\varepsilon>0$, there exists $j_0\ge 1$ such that for any $j\ge j_0$,
\[\|\phi(y_j)-\phi(y)\|<\varepsilon.\] 
In particular, if $j\ge j_0$ and $a\in L_{\phi(y_j)}$ then $d_{\infty}(L_{\phi(y)},a)\le \varepsilon$. We get the desired conclusion since $L_{\phi(y_j)}\supset \{\frac{k}{2^j}:0\le k\le 2^j-1\}$.
\end{proof}

Although the set $L(\omega)\cup\{0\}$ does not contain an isometric copy of $\ell^{\infty}$, we can succeed to construct a  closed non-separable subspace in $L(\omega)\cup\{0\}$ by using another approach.

\begin{prop}
The set $L(\omega)\cup\{0\}$ contains a closed non-separable subspace.
\end{prop}

\begin{proof}
  Let $(A_s)_{s\in \mathbb{R}}$ be a family of infinite subsets of $\mathbb{N}$ such that for any $s\ne t$, the intersection $A_s\cap A_t$ is finite. Such a family can be found in \cite[Lemma 2.1]{Dim}. For each $s\in \mathbb{R}$, we consider a partition of $A_s$ into infinite subsets $(A_s^m)_{m=0}^{\infty}$. If $c_m=\frac{1}{2^m}$, $m\geq 0$, we let $(x_s)_{s\in \mathbb{R}}\subset L(\omega)$ by
  $$
  x_{s,n}=\begin{cases}
      &c_m \,\, \text{if} \,\, n\in A_s^m,\\
      &0 \,\,\,\ \ \text{otherwise.}
  \end{cases}
  $$
  We notice that $L_{x_s}=\{c_m: m\geq 0\}\cup \{0\}$. Let
  $$
  M=\overline{\text{span}}\{x_s: s\in \mathbb{R}\}.
  $$
  It is clear that $M$ is closed and non-separable since it contains the uncountable family $(x_s)_{s\in \mathbb{R}}$ and that for any $s\ne t$, $\|x_s-x_t\|=1$. 

  We will show that $M\subset L(\omega)\cup\{0\}$. If $x\in M\backslash\{0\}$, then there is a sequence $(y_N)_{N=1}^{\infty}\subset \text{span}\{x_s: s\in \mathbb{R}\}$ such that $y_N\rightarrow x$. We may then find a sequence $(s_k)_{k\ge 0}$ of distinct real numbers such that
  $$
  y_N=\sum_{k=0}^{\infty}a^N_kx_{s_k}
  $$
  where $a^N_k\in \mathbb{R}$ and where for any $N$, only finitely many coefficients $a^N_k$ are non-zero. We also observe that 
  $$
  L_{y_N}=\{a^N_kc_m: k,m \geq 0\}\cup\{0\}.
  $$
 By the fact that the coordinates of $y_N$ are eventually equal to $a^N_k$ along the set $A^0_{s_k}$ and since $(y_N)_{N=1}^{\infty}$ is a Cauchy sequence, we get that if $\varepsilon >0$, there exists $N_0\in \mathbb{N}$, such that for each $N_1,N_2\geq N_0$
  $$
  |a^{N_1}_k-a^{N_2}_k|\leq \|y_{N_1}-y_{N_2}\|<\varepsilon,
  $$
  for each $k\geq 0$. Thus, there are $a_k\in \mathbb{R}$ ($k\geq 0$) such that 
   $$
   a^N_k\rightarrow a_k \quad\text{as $N\rightarrow \infty$, uniformly for $k\geq 0$.}
   $$
   In particular, there is $L>0$ such that $|a^N_k|\leq L$ for all $N\geq 1$ and all $k\geq 0$. 

   We next prove the fact that for every $\varepsilon >0$, there are $N_0\in \mathbb{N}$, $K\geq 0$ and $M\geq 0$, such that if for some $N\geq N_0$ it holds that $|y_{N,n}|\geq \varepsilon$, then $n\in \bigcup_{k\leq K}\bigcup_{m\leq M}A^m_{s_k}$. Indeed, again by the fact that $(y_N)_{N=1}^{\infty}$ is a Cauchy sequence, we get $N_0\in \mathbb{N}$ such that for each $N\geq N_0$, $\|y_N-y_{N_0}\|<\frac{\varepsilon}{2}$. Let 
   $$
   K=\max\{k\geq 0: a^{N_0}_k\neq 0\}
   $$
   and $M\geq 0$ such that
   $$
   2c_{M+1}LK<\varepsilon.
   $$
   Let $n$ such that $|y_{N,n}|\geq \varepsilon$ for some $N\ge N_0$. If $n\notin \bigcup_{k\leq K}\bigcup_{m\leq M}A^m_{s_k}$ then 
   \[|y_{N_0,n}|\le \sum_{k=1}^K |a^N_k| c_{M+1} \leq KLc_{M+1}<\frac{\varepsilon}{2}\]
   and thus
   \[|y_{N,n}| <\|y_{N}-y_{N_0}\|+\frac{\varepsilon}{2}\le \varepsilon\]
   for all $N\geq N_0$ which is a contradiction.

   We now show that 
   $$
   L_x=\{a_kc_m: k,m\geq 0\} \cup \{0\}
   $$
   and thus that $x\in L(\omega)$.
   Let $a\in L_x$ with $a\ne 0$. We consider an increasing sequence $(n_l)$ such that
   $$
   \lim_{l\to \infty}x_{n_l}=a.
   $$
   Since $y_N\rightarrow x$, it holds that there is $N_0\in \mathbb{N}$ such that for every $N\ge N_0$, $|y_{N,n_l}|>\frac{|a|}{2}$ eventually. The fact of the previous paragraph applied for $\varepsilon=\frac{|a|}{2}$ gives us that there are $K,M\geq 0$ such that
   $$
   (n_l)\subset \bigcup_{k\leq K}\bigcup_{m\leq M}A^m_{s_k}
   $$
   eventually. Passing to a subsequence, we can assume that 
   $$
   (n_l)\subset A^m_{s_k}
   $$
   for some $k\leq K$ and $m\leq M$. We now remark that $(y_{N,n_l})_N$ converges to $x_{n_l}$ uniformly in $l$ and that for any $N$, $(y_{N,n_l})_l$ converges to $a^N_kc_m$. It follows that
   $$
   a=\lim_{l\rightarrow \infty}x_{n_l}=\lim_{l\rightarrow \infty} \lim_{N\rightarrow \infty}y_{N,n_l}=\lim_{N\rightarrow \infty}\lim_{l\rightarrow \infty}y_{N,n_l}= \lim_{N\rightarrow \infty}a^N_kc_m=a_kc_m.
   $$
   That gives us that $L_x\subset \{a_kc_m: k,m\geq 0\}\cup \{0\}$. 

   The reverse inclusion is even easier. For each $k,m\geq 0$, it suffices to consider $(n_l)$ the increasing enumeration of $A_{s_k}^m$ and to deduce as above that $\lim_{l\rightarrow \infty}x_{n_l}=a_kc_m$. Finally, since $L_x$ is closed and $c_m$ tends to $0$, we conclude that $L_x= \{a_kc_m: k,m\geq 0\}\cup \{0\}$.

\end{proof}

\end{document}